\documentclass[reqno]{amsart}
\usepackage{amsmath}
\usepackage{amssymb}
\usepackage{amsthm}
\begin{document}
\def\eq#1{{\rm(\ref{#1})}}
\theoremstyle{plain}
\newtheorem*{theo}{Theorem}
\newtheorem*{ack}{Acknowledgements}
\newtheorem*{pro}{Proposition}
\newtheorem*{coro}{Corollary}
\newtheorem{thm}{Theorem}[section]
\newtheorem{lem}[thm]{Lemma}
\newtheorem{prop}[thm]{Proposition}
\newtheorem{cor}[thm]{Corollary}
\theoremstyle{definition}
\newtheorem{dfn}[thm]{Definition}
\newtheorem*{rem}{Remark}
\def\coker{\mathop{\rm coker}}
\def\ind{\mathop{\rm ind}}
\def\Re{\mathop{\rm Re}}
\def\vol{\mathop{\rm vol}}
\def\Im{\mathop{\rm Im}}
\def\im{\mathop{\rm im}}
\def\Hol{{\textstyle\mathop{\rm Hol}}}
\def\C{{\mathbin{\mathbb C}}}
\def\R{{\mathbin{\mathbb R}}}
\def\N{{\mathbin{\mathbb N}}}
\def\Z{{\mathbin{\mathbb Z}}}
\def\O{{\mathbin{\mathbb O}}}
\def\L{{\mathbin{\mathcal L}}}
\def\X{{\mathbin{\mathcal X}}}
\def\al{\alpha}
\def\be{\beta}
\def\ga{\gamma}
\def\de{\delta}
\def\ep{\epsilon}
\def\io{\iota}
\def\ka{\kappa}
\def\la{\lambda}
\def\ze{\zeta}
\def\th{\theta}
\def\vt{\vartheta}
\def\vp{\varphi}
\def\si{\sigma}
\def\up{\upsilon}
\def\om{\omega}
\def\De{\Delta}
\def\Ga{\Gamma}
\def\Th{\Theta}
\def\La{\Lambda}
\def\Om{\Omega}
\def\Up{\Upsilon}
\def\sm{\setminus}
\def\na{\nabla}
\def\pd{\partial}
\def\op{\oplus}
\def\ot{\otimes}
\def\bigop{\bigoplus}
\def\iy{\infty}
\def\ra{\rightarrow}
\def\longra{\longrightarrow}
\def\dashra{\dashrightarrow}
\def\t{\times}
\def\w{\wedge}
\def\bigw{\bigwedge}
\def\d{{\rm d}}
\def\bs{\boldsymbol}
\def\ci{\circ}
\def\ti{\tilde}
\def\ov{\overline}
\def\sv{\star\vp}
\title[Closed $G_2$-Structures]{Diffeomorphisms of $7$-Manifolds with Closed $G_2$-Structure}

\author[Cho, Salur, and Todd]{Hyunjoo Cho, Sema Salur, and A. J. Todd}

\address {Department of Mathematics, University of Rochester, Rochester, NY, $14627$}
\email{cho@math.rochester.edu}

\address {Department of Mathematics, University of Rochester, Rochester, NY, 14627}
\email{salur@math.rochester.edu}

\address {Department of Mathematics, University of California - Riverside, Riverside, CA, 92521}
\email{ajtodd@math.ucr.edu}

\begin{abstract}
We introduce $G_2$-vector fields, Rochesterian $1$-forms and Rochesterian vector fields on manifolds with a closed $G_2$-structure as analogues of symplectic vector fields, Hamiltonian functions and Hamiltonian vector fields respectively, and we show that the spaces $\X_{G_2}$ of $G_2$-vector fields and $\X_{Roc}$ of Rochesterian vector fields are Lie subalgebras of the Lie algebra of vector fields with the standard Lie bracket. We also define, in analogy with the Poisson bracket on smooth real-valued functions from symplectic geometry, a bracket operation on the space of Rochesterian $1$-forms $\Om^1_{Roc}$ associated to the space of Rochesterian vector fields and prove, despite the lack of a Jacobi identity, a relationship between this bracket and diffeomorphisms which preserve $G_2$-structures.
\end{abstract}

\date{}
\maketitle
\section*{Introduction}
The possible holonomy groups for a given $7$-dimensional Riemannian manifold include the exceptional Lie group $G_2$ by Berger's classification of Riemannian holonomy groups. Such manifolds are called \emph{$G_2$-manifolds} and are equipped with a nondegenerate differential $3$-form $\vp$ which is torsion-free, $\na\vp=0$, with respect to the Levi-Civita connection of the metric $g_{\vp}$ defined by $\vp$. This torsion-free condition is equivalent to $\vp$ being closed and coclosed, see \cite{Jo} or \cite{Sa}. Much work has been done to study manifolds with $G_2$-holonomy, e. g., \cite{Br1}, \cite{BrSa} and \cite{Jo}, but the condition $\vp$ be coclosed, $\d^*\vp=0$, is a nonlinear condition since $\d^*$ depends on the Hodge star given by the metric $g_{\vp}$ above. If we drop this coclosed condition, then we have a manifold with a \emph{closed $G_2$-structure}. Manifolds with closed $G_2$-structures have been studied in many articles including \cite{Br2}, \cite{BrXu} and \cite{ClIv}; however, these papers however focused predominantly on the $G_2$ metric itself defined by the nondegenerate closed $3$-form $\vp$. We shift our focus to the form $\vp$ and to results which depend on $\vp$ being nondegenerate and closed; this article is a continuation of a project which began with \cite{ACS} to better understand $G_2$-geometry by using the well-established areas of symplectic and contact geometry.

Treating symplectic geometry and $G_2$ geometry as analogues is not new. \cite{BrGr} and \cite{Gray} study vector cross products on linear spaces and on manifolds; further, one can construct, using a metric, a nondegenerate differential form of degree $k+1$ associated to a $k$-fold vector cross product. In particular, it is shown that associated to a $1$-fold vector cross product, i. e., an almost complex structure, there is a nondegenerate $2$-form, which, when this form is closed, yields a symplectic form; in an analogous way, we can view $G_2$ geometry as the geometry of $2$-fold vector cross products in dimension $7$ ($2$-fold vector cross products only exist in dimensions $3$ and $7$ as is shown in the above mentioned articles). In particular, we get the fundamental $G_2$ $3$-forms $\vp$ associated to these $2$-fold vector cross products in dimension $7$. Examples of manifolds with $G_2$-structures satisfying various conditions (including closed $G_2$-structures) are studied, for example, in \cite{CMS}, \cite{Fe1}, \cite{Fe2} and \cite{FeIg}, and classified in \cite{FeGr}. Links between Calabi-Yau geometry and $G_2$ geometry in the context of mirror symmetry have been actively explored by many mathematicians and theoretical physicists over the course of many articles including, for example, works by Atiyah-Witten \cite{AtWi}, Akbulut-Salur \cite{AkSa}, Gukov-Yau-Zaslow \cite{GYZ}, Lee-Leung \cite{LeLe} and Leung \cite{Leun}.

From a different perspective, there is multisymplectic geometry. This has been, and continues to be, a very active area of research. Multisymplectic geometry is the study of a smooth manifold of arbitrary dimension equipped with a nondegenerate closed $(n+1)$ form. This field grew from the fact that many results in symplectic geometry depend only on the fact that the symplectic form is nondegenerate and closed. Background references for this area include, e. g., \cite{CCI}, \cite{CIdL2}. Further, multisymplectic geometry provides a natural setting for many questions in classical Hamiltonian mechanics, e. g., \cite{CIdL1}, \cite{PR} as well as questions regarding string theory \cite{BHR}, \cite{BR}. Our work in this article then can been viewed as a specific case of multisymplectic geometry where we are considering a nondegenerate closed $3$-form on a $7$-dimensional manifold, and there is much to be gained by looking at this specific case as is hinted at in the current paper (see, e. g., Theorem $3.2$); moreover, these ideas have already led to several results not obtainable in the general multisymplectic setting and form the basis for a number of upcoming articles in $G_2$ and $Spin(7)$ geometry. Indeed, multisymplectic geometry is difficult because the assumptions are so general; on the other hand, the geometric structures associated with the exceptional Lie group $G_2$ are completely dependent on one another (see Section $2$) and so provide a rich setting for inquiry. Further, because of the large interest in M-theory, there is a need for results specifically related to $G_2$ geometry.

This paper consists of three sections: the first section is a review of ideas from symplectic geometry. We discuss symplectic and Hamiltonian vector fields and show that symplectic vector fields, Hamiltonian vector fields and smooth real-valued functions on $M$ all admit the structure of Lie algebras with Lie bracket on the symplectic and Hamiltonian vector fields induced from the Lie bracket structure on the space of all vector fields and the Lie bracket on smooth real-valued functions on $M$ given by a Poisson bracket; further, there is a Lie algebra anti-homomorphism between the Lie algebra of smooth real-valued functions on $M$ and the Lie algebra of the Hamiltonian vector fields.

We give a brief introduction to $G_2$ geometry in the second section and prove the following standard result regarding diffeomorphisms which preserve $G_2$ structures which has not previously appeared in the literature:

\begin{theo}
Let $(M_1,\vp_1)$ and $(M_2,\vp_2)$ be two manifolds with $G_2$-structures $\vp_1$, $\vp_2$ respectively. Let $\pi_i:M_1\t M_2\to M_i$ be the standard projection map, and define a $3$-form on the product manifold $M_1\t M_2$ by $\tilde{\vp}:=\pi_1^*\vp_1-\pi_2^*\vp_2$. A diffeomorphism $\Up:(M_1,\vp_1)\to(M_2,\vp_2)$ is a $G_2$-morphism if and only if $\tilde{\vp}|_{\Ga_{\Up}}\equiv 0$, where $\Ga_{\Up}:=\{(p,\Up(p))\in M_1\t M_2:p\in M_1\}$ is the graph of $\Up$ in $M_1\t M_2$.
\end{theo}

In the third section, we define the analogues of Hamiltonian functions, Hamiltonian vector fields and symplectic vector fields given by Rochesterian $1$-forms, Rochesterian vector fields and $G_2$-vector fields respectively then prove the following results for Rochesterian vector fields:

\begin{theo}
There are no nontrivial Rochesterian vector fields on a closed manifold $M$ with closed $G_2$-structure $\vp$.
\end{theo}

\begin{theo}
Every Rochesterian vector field on a manifold $M$ with closed $G_2$-structure $\vp$ is a $G_2$-vector field. If every closed form in $\Om^2_7(M)$ is exact, then the spaces $\X_{Roc}(M)$ and $\X_{G_2}(M)$ coincide.
\end{theo}

\begin{coro}
If $H^2(M)=\{0\}$, then every $G_2$-vector field on a manifold with closed $G_2$-structure is a Rochesterian vector field.
\end{coro}

We next show that the spaces of $G_2$- and Rochesterian vector fields admit the structure of Lie algebras with Lie bracket induced from the standard Lie bracket structure on the space of all vector fields and prove the following result on inclusions:

\begin{pro}
For any $G_2$-vector fields $X_1$, $X_2$, $[X_1,X_2]$ is a Rochesterian vector field with associated Rochesterian $1$-form given by $\vp(X_2,X_1,\cdot)$.
\end{pro}

Finally, we equip the space of Rochesterian $1$-forms with a bracket structure analogous to that of the Poisson bracket from symplectic geometry, show that it does \emph{not} satisfy the Jacobi identity, show that there is a linear transformation $\Phi$ of the \emph{vector spaces} of Rochesterian $1$-forms and Rochesterian vector fields and prove the following result regarding these structures:

\begin{theo}
\begin{enumerate}
    \item Given two Rochesterian $1$-forms $\al_1,\al_2\in\Om^1_{Roc}(M)$, $\{\al_1,\al_2\}\in\ker\Phi$ if and only if $\d\al_1$ is constant along the flow lines of $X_{\al_2}$ if and only if $\d\al_2$ is constant along the flow lines of $X_{\al_1}$.
    \item Let $\psi:(M,\vp)\to(M',\vp')$ be a diffeomorphism. Then $\psi$ is a $G_2$-morphism if and only if $\psi^*(\{\al,\be\})=\{\psi^*\al,\psi^*\be\}$ for all $\al,\be\in\Om^1_{Roc}(M')$.
\end{enumerate}
\end{theo}

\begin{ack}
Many many thanks to several anonymous referees whose comments, corrections and suggestions helped to significantly improve this paper. We sincerely appreciate all of your insights, your help and your time.
\end{ack}

\section{Symplectic Vector Fields and Hamiltonian Vector Fields}
The material in this section comes from \cite{McSa} and \cite{daSi}. Let $(M,\om)$ be an arbitrary $2n$-dimensional manifold with a closed, nondegenerate $2$-form $\om$.

\begin{dfn}
\begin{enumerate}
    \item A vector field $X$ is called a \emph{symplectic vector field} if the flow induced by $X$ preserves the symplectic form $\om$, that is, $X$ is a symplectic vector field if and only if $\L_X\om=0$.
    \item A vector field $X$ is called a \emph{Hamiltonian vector field} if there exists a smooth real-valued function $H$ on $M$ such that $X\lrcorner\om=\d H$.
\end{enumerate}
\end{dfn}

Notice that since $\om$ is closed, we have, by the Cartan Formula, $\L_X\om=\d(X\lrcorner\om)+X\lrcorner\d\om=\d(X\lrcorner\om)$ for any vector field $X$, so $X$ is symplectic if and only if the $1$-form $X\lrcorner\om$ is closed. Hamiltonian vector fields always exist since $\om$ induces an isomorphism of smooth sections of the tangent bundle with smooth sections of the cotangent bundle: Given any smooth function $H$, $\d H$ is a covector field, so there exists a unique vector field $X_H$ such that $X_H\lrcorner\om=\d H$. An immediate consequence of these definitions is that a Hamiltonian vector field $X$ is always symplectic for if $X\lrcorner\om=\d H$ for some smooth real-valued function $H$, then $\d(X\lrcorner\om)=\d(\d H)=0$. The converse is not true in general; in fact, the obstruction for a symplectic vector field $X$ on $(M,\om)$ to be Hamiltonian is $H^1(M)$. Indeed, if $H^1(M)=\{0\}$, i. e., every closed $1$-form is exact, then for a symplectic vector field $X$, there exists a smooth real-valued function $H$ on $M$ such that $X\lrcorner\om=\d H$, that is, $X$ is a Hamiltonian vector field.

Let $\X(M)$ denote the space of vector fields on $M$, $\X_{symp}(M)$ the subspace of symplectic vector fields on $M$ and $\X_{Ham}(M)$ the subspace of Hamiltonian vector fields on $M$, and equip $\X(M)$ with the standard Lie bracket $[X,Y]=XY-YX$. Then:

\begin{prop}
If $X_1,X_2$ are symplectic vector fields, then the Lie bracket $[X_1,X_2]$ is a Hamiltonian vector field.
\end{prop}

\begin{proof}
Recall that for an arbitrary differential form $\tau$ we have $[X,Y]\lrcorner\tau=\L_X(Y\lrcorner\tau)-Y\lrcorner(\L_X\tau)$, so
\begin{equation*}
\begin{split}
[X_1,X_2]\lrcorner\om&=\L_{X_1}(X_2\lrcorner\om)-X_2\lrcorner(\underbrace{\L_{X_1}\om}_{=0}) \\
&=\L_{X_1}(X_2\lrcorner\om)=\d(X_1\lrcorner X_2\lrcorner\om)+X_1\lrcorner(\underbrace{\d(X_2\lrcorner\om)}_{=0}) \\
&=\d(\om(X_2,X_1))
\end{split}
\end{equation*}
Hence $[X_1,X_2]$ is a Hamiltonian vector field with generating Hamiltonian function $\om(X_2,X_1)$.
\end{proof}

\begin{cor}
The subspaces $\X_{symp}(M)$ and $\X_{Ham}(M)$ of $\X(M)$ are closed under the Lie bracket operation inherited from $\X(M)$; hence, there are the following inclusions of \emph{Lie} algebras: $$(\X_{Ham}(M),[\cdot,\cdot])\subseteq(\X_{symp}(M),[\cdot,\cdot])\subseteq(\X(M),[\cdot,\cdot]).$$
\end{cor}

We now focus on the real-valued smooth functions on $M$, $C^{\infty}(M)$. For $f\in C^{\infty}(M)$, the assignment $f\mapsto X_f$ where $X_f$ is the associated Hamiltonian vector field is linear. Given $f,g\in C^{\infty}(M)$, then $$(X_f+X_g)\lrcorner\om=(X_f\lrcorner\om)+(X_g\lrcorner\om)=\d f+\d g=\d(f+g)=X_{f+g}\lrcorner\om,$$ so that by nondegeneracy of $\om$, we have $X_{f+g}=X_f+X_g$. Similarly, $X_{af}=aX_f$. We now equip $C^{\infty}(M)$ with a bracket operation as follows: For $f,g\in C^{\infty}(M)$, define $\{f,g\}=\om(X_f,X_g)\in C^{\infty}(M)$. Consider the Hamiltonian vector field $X_{\{f,g\}}$: $$X_{\{f,g\}}\lrcorner\om=X_{\om(X_f,X_g)}\lrcorner\om=([X_g,X_f])\lrcorner\om,$$ so that $X_{\{f,g\}}=-[X_f,X_g]$.

\begin{prop}
The bracket $\{\cdot,\cdot\}$ on $C^{\infty}(M)$ satisfies the Jacobi identity.
\end{prop}

\begin{proof}
Let $f,g,h\in C^{\infty}(M)$ with associated Hamiltonian vector fields $X_f$, $X_g$ and $X_h$ respectively. Then we have the following:
\begin{equation*}
\begin{split}
\{f&,\{g,h\}\}+\{g,\{h,f\}\}+\{h,\{f,g\}\}=\{f,\{g,h\}\}-\{g,\{f,h\}\}-\{\{f,g\},h\}\\
&=X_f\lrcorner X_{\{g,h\}}\lrcorner\om-X_g\lrcorner X_{\{f,h\}}-X_{\{f,g\}}\lrcorner X_h\lrcorner\om\\
&=X_f\lrcorner\d\{g,h\}-X_g\lrcorner\d\{f,h\}+[X_f,X_g]\lrcorner\d h\\
&=X_f\lrcorner\d(X_h\lrcorner X_g\lrcorner\om)-X_g\lrcorner\d(X_h\lrcorner X_f\lrcorner\om)+[X_f,X_g]\lrcorner\d h\\
&=-X_f\lrcorner\d(X_g\lrcorner X_h\lrcorner\om)+X_g\lrcorner\d(X_f\lrcorner X_h\lrcorner\om)+[X_f,X_g]\lrcorner\d h\\
&=-X_f\lrcorner\d(X_g\lrcorner\d h)+X_g\lrcorner\d(X_f\lrcorner\d h)+[X_f,X_g]\lrcorner\d h\\
&=-X_f\lrcorner\d(X_g\lrcorner\d h)+X_g\lrcorner\d(X_f\lrcorner\d h)+\L_{X_f}(X_g\lrcorner\d h)-X_g\lrcorner(\L_{X_f}\d h)\\
&=-X_f\lrcorner\d(X_g\lrcorner\d h)+X_g\lrcorner\d(X_f\lrcorner\d h)+X_f\lrcorner\d(X_g\lrcorner\d h)+\underbrace{\d(X_f\lrcorner X_g\lrcorner\d h)}_{=0}\underbrace{-X_g\lrcorner(X_f\lrcorner\d\d h)}_{=0}\\
&-X_g\lrcorner\d(X_f\lrcorner\d h)\\
&=X_f\lrcorner\d(X_g\lrcorner\d h)-X_f\lrcorner\d(X_g\lrcorner\d h)+X_g\lrcorner\d(X_f\lrcorner\d h)-X_g\lrcorner\d(X_f\lrcorner\d h)=0
\end{split}
\end{equation*}
\end{proof}

Hence, $(C^{\infty}(M),\{\cdot,\cdot\})$ is a Lie algebra, and there is a Lie algebra anti-homomorphism $\Psi:(C^{\infty}(M),\{\cdot,\cdot\})\to(\X_{Ham},[\cdot,\cdot])$ given by $f\mapsto X_f$. Assume that $\Psi(f)=X_f=0$, then $0=X_f\lrcorner\om=\d f$ implies that $f$ is locally constant (or constant if $M$ is connected); therefore, $\ker\Psi=\{\text{(locally) constant functions on }M\}$, so every Hamiltonian vector field is defined by a smooth real-valued function on $M$ which is unique up to the addition of a locally constant smooth function.

\begin{thm}
\begin{enumerate}
    \item For $f,g\in C^{\infty}(M)$, $\{f,g\}=0$ if and only if $f$ is constant along the integral curves determined by $X_g$ if and only if $g$ is constant along the integral curves determined by $X_f$.
    \item Let $\psi:(M,\om)\to(M',\om')$ be a diffeomorphism. Then $\psi$ is a symplectomorphism if and only if $\{f,g\}\circ\psi=\{f\circ\psi,g\circ\psi\}$ for all $f,g\in C^{\infty}(M')$.
\end{enumerate}
\end{thm}

\begin{proof}
\begin{enumerate}
    \item We show only the first equivalence since the second equivalence follows similarly. Let $\psi_t$ denote the integral curves generated by $X_g$. The result then follows immediately from the following calculation.
    \begin{equation*}
    \begin{split}
    \frac{\d}{\d t}(f\circ\psi_t)&=\psi_t^*\L_{X_g}f=\psi_t^*(X_g\lrcorner\d f) \\
    &=\psi_t^*(X_g\lrcorner X_f\lrcorner\om)=\psi_t^*\om(X_f,X_g)=\psi_t^*\{f,g\}
    \end{split}
    \end{equation*}
    \item Assume first that $\psi$ is a symplectomorphism. Note that for $p\in M$, we have the maps
    \begin{equation*}
    \begin{split}
    &\d\psi_p:T_pM\to T_{\psi(p)}M'\\
    &\psi^*_p:T^*_{\psi(p)}M'\to T^*_pM\\
    &\d\psi^{-1}_{\psi(p)}=(\d\psi_p)^{-1}:T_{\psi(p)}M'\to T_pM
    \end{split}
    \end{equation*}
    Since $\psi^*\om'=\om$ and $(\psi^{-1})^*\om=\om'$, we have the following equivalent equations:
    \begin{equation*}
    \begin{split}
    &\om_p(\cdot,\cdot)=\psi^*_p(\om'_{\psi(p)})(\cdot,\cdot)=\om'_{\psi(p)}(\d\psi_p\cdot,\d\psi_p\cdot)\\
    &\om'_{\psi(p)}(\cdot,\cdot)=(\psi^{-1}_{\psi(p)})^*(\om_p)(\cdot,\cdot)=\om_p(\d\psi^{-1}_{\psi(p)}\cdot,\d\psi^{-1}_{\psi(p)}\cdot)
    \end{split}
    \end{equation*}
    \noindent
    For a function $f\in C^{\infty}(M')$ and vector field $Y\in\X(M)$, we now calculate
    \begin{equation*}
    \begin{split}
    (X_{f\circ\psi}\lrcorner\om)_p(Y_p)&=\d(f\circ\psi)_p(Y_p)=\d f_{\psi(p)}(\d\psi_pY_p)=\psi^*_p(\d f_{\psi(p)})(Y_p)\\ &=\psi^*_p((X_f\lrcorner\om')_{\psi(p)})(Y_p)=\psi^*_p(\om'_{\psi(p)}((X_f)_{\psi(p)},\cdot))(Y_p)\\
    &=\om'_{\psi(p)}((X_f)_{\psi(p)},\d\psi_pY_p)\\
    &=\om_p(\d\psi^{-1}_{\psi(p)}(X_f)_{\psi(p)},\d\psi^{-1}_{\psi(p)}(\d\psi_pY_p))\\
    &=\om_p((\d\psi^{-1})_{\psi(p)}(X_f)_{\psi(p)},Y_p)\\
    \end{split}
    \end{equation*}
    that is, $(X_{f\circ\psi})_p=(\d\psi^{-1})_{\psi(p)}(X_f)_{\psi(p)}$. Hence we find that
    \begin{equation*}
    \begin{split}
    (\{f,g\}\circ\psi)(p)&=\om'(X_f,X_g)(\psi(p))\\
    &=\om'_{\psi(p)}((X_f)_{\psi(p)},(X_g)_{\psi(p)})\\ &=\om_p(\d\psi^{-1}_{\psi(p)}(X_f)_{\psi(p)},\d\psi^{-1}_{\psi(p)}(X_g)_{\psi(p)})\\
    &=\om_p((X_{f\circ\psi})_p,(X_{g\circ\psi})_p)\\
    &=\om(X_{f\circ\psi},X_{g\circ\psi})(p)=\{f\circ\psi,g\circ\psi\}(p)
    \end{split}
    \end{equation*}
    Conversely, assume that $\{f,g\}\circ\psi=\{f\circ\psi,g\circ\psi\}$ for all $f,g\in C^{\infty}(M')$. Then, for any $f,g\in C^{\infty}(M')$ we have 
    \begin{equation*}
    \begin{split}
    \{f,g\}\circ\psi&=\om'(X_f,X_g)\circ\psi=(X_g\lrcorner X_f\lrcorner\om')\circ\psi=(X_g\lrcorner\d f)\circ\psi\\
    &=(\d f(X_g))\circ\psi=\psi^*(\d f(X_g))=\psi^*(X_gf)=(\d\psi^{-1}X_g)(f\circ\psi)\\
    \end{split}
    \end{equation*}
    and
    \begin{equation*}
    \begin{split}
    \{f\circ\psi,g\circ\psi\}&=\om(X_{f\circ\psi},X_{g\circ\psi})=X_{g\circ\psi}\lrcorner\d(f\circ\psi)\\
    &=\d(f\circ\psi)(X_{g\circ\psi})=X_{g\circ\psi}(f\circ\psi)
    \end{split}
    \end{equation*}
    which, by our hypothesis, yields that $\d\psi^{-1}X_g=X_{g\circ\psi}$ for any $g\in C^{\infty}(M')$. Then for any $f\in C^{\infty}(M')$ and any vector field $Y\in\X(M)$,
    \begin{equation*}
    \begin{split}
    (X_{f\circ\psi}\lrcorner\om)_p(Y_p)&=\d(f\circ\psi)_p(Y_p)=\d(\psi^*f)_p(Y_p)=(\psi^*\d f)_p(Y_p)\\
    &=\psi^*_p(X_f\lrcorner\om')_p(Y_p)=\om'_{\psi(p)}((X_f)_{\psi(p)},\d\psi_pY_p)\\
    &=\om'_{\psi(p)}(\d\psi_p(\d\psi^{-1}_{\psi(p)}(X_f)_{\psi(p)}),\d\psi_pY_p)\\
    &=(\psi^*\om')_{\psi(p)}(\d\psi^{-1}_{\psi(p)}(X_f)_{\psi(p)},Y_p)\\
    &=(\psi^*\om')_{\psi(p)}((X_{f\circ\psi})_p,Y_p)
    \end{split}
    \end{equation*}
    \noindent
    Thus $X_{f\circ\psi}\lrcorner\om=X_{f\circ\psi}\lrcorner\psi^*\om'$ which implies that $\om=\psi^*\om'$ as desired.
\end{enumerate}
\end{proof}

\section{$G_2$ Geometry}
If we consider coordinates $(x_1,\ldots,x_7)$ on $\R^7$, we can define a $3$-form $\vp_0$ by $$\vp_0=\d x^{123}+\d x^{145}+\d x^{167}+\d x^{246}-\d x^{257}-\d x^{347}-\d x^{356}.$$ From this $3$-form, we get an induced metric and orientation by the formula $$(X\lrcorner\vp_0)\w(Y\lrcorner\vp_0)\w\vp_0=6<X,Y>_{\vp_0}\d vol_{\vp_0}$$ for vector fields $X,Y\in \X(\R^7)$. Then, using this metric, we can define a $2$-fold vector cross product of $X$ and $Y$ as the unique vector field $X\t Y$ satisfying $<X\t Y,Z>_{\vp_0}=\vp_0(X,Y,Z)$ for all $Z\in\X(\R^7)$. This metric then gives the associated Hodge star from which we get the dual of $\vp_0$ given by the $4$-form $$\star\vp_0=\d x^{4567}+\d x^{2367}+\d x^{2345}+\d x^{1357}-\d x^{1346}-\d x^{1256}-\d x^{1247}.$$ Here, we have started with the $3$-form and have shown how to define the other structures in terms of it; a fundamental fact of $G_2$-geometry however is that given \emph{any one} of $\vp_0$, $\star\vp_0$, $\t$ or $<,>$, we can always define the other structures. References for this information and an equivalent formulation of these structures arising from the octonions include \cite{BrGr, Br1, FeGr, Gray, HaLa, Jo, Kari}.

\begin{dfn}
A manifold $M$ is said to have a \emph{$G_2$-structure} if there is a $3$-form $\vp$ such that $(T_pM,\vp)\cong(\R^7,\vp_0)$ as vector spaces for every point $p\in M$. This is equivalent to a reduction of the tangent frame bundle from $GL(7,\R)$ to the Lie group $G_2$. If $\d\vp=0$, then the $G_2$-structure is said to be \emph{closed}.
\end{dfn}

\begin{rem}
Because of the inclusion $G_2$ in $SO(7)$, all manifolds with $G_2$-structure are necessarily orientable; in addition, it can be shown that all manifolds with $G_2$-structure are spin, and any $7$-manifold with spin structure admits a $G_2$-structure.
\end{rem}

A natural geometric requirement is that $\vp$ be constant with respect to the Levi-Civita connection of the $G_2$-metric $g_{\vp}$ defined by $\vp$. In this case, the holonomy of $(M,\vp)$ is a subgroup of $G_2$, and $(M,\vp)$ is called a \emph{$G_2$-manifold}. The condition that $\na\vp=0$ is equivalent to $\d\vp=0$ and $\d^*\vp=0$ where $\d^*$ is the adjoint operator to the exterior derivative with respect to the Hodge star associated to the $G_2$-metric $g_{\vp}$. Fernandez and Gray \cite{FeGr} show that manifolds with closed $G_2$-structure and $G_2$-manifolds are just $2$ of $16$ types of $G_2$-structures on manifolds.

\begin{dfn}
Let $(M_1,\vp_1)$ and $(M_2,\vp_2)$ be $7$-manifolds with $G_2$-structures. If $\Up:M_1 \to M_2$ is a diffeomorphism such that $\Up^*(\vp_2)=\vp_1$, then $\Up$ is called a \emph{$G_2$-morphism} and $(M_1,\vp_1)$, $(M_2,\vp_2)$ are said to be \emph{$G_2$-morphic}.
\end{dfn}
\noindent
Notice that given a $G_2$-morphism $\Up:(M_1,\vp_1)\to(M_2,\vp_2)$, $\d\vp_1=0$ if and only if $\d\vp_2=0$ since $\d$ commutes with pullback maps.

Let $(M_1,\vp_1)$ and $(M_2,\vp_2)$ be two $7$-dimensional manifolds with $G_2$-structures. Let $M_1\t M_2$ be the standard Cartesian product of $M_1$ and $M_2$ with canonical projection maps $\pi_i:M_1\t M_2\to M_i$. Define a $3$-form $\vp=\pi_1^*\vp_1+\pi_2^*\vp_2$. If both $\vp_1$ and $\vp_2$ are closed, then this form is also closed; in fact, for any $a_1,a_2\in\R$, $a_1\pi_1^*\vp_1+a_2\pi_2^*\vp_2$ defines a (closed) $3$-form on $M_1\t M_2$. Taking $a_1=1$ and $a_2=-1$, we have the (closed) $3$-form $\tilde{\vp}=\pi_1^*\vp_1-\pi_2^*\vp_2$.

\begin{thm}
A diffeomorphism $\Up:(M_1,\vp_1)\to(M_2,\vp_2)$ is a $G_2$-morphism if and only if $\tilde{\vp}|_{\Ga_{\Up}}\equiv 0$, where $\Ga_{\Up}:=\{(p,\Up(p))\in M_1\t M_2:p\in M_1\}$.
\end{thm}

\begin{proof}
The submanifold $\Ga_{\Up}$ is the embedded image of $M_1$ in $M_1\t M_2$ with embedding given by $\tilde{\Up}:M_1\to M_1\t M_2$ $\tilde{\Up}(p)=(p,\Up(p))$. Then $\tilde{\vp}|_{\Ga_{\Up}}=0$ if and only if $$0=\tilde{\Up}^*\tilde{\vp}= \tilde{\Up}^*\pi_1^*\vp_1-\tilde{\Up}^*\pi_2^*\vp_2$$ $$=(\pi_1\circ\tilde{\Up})^*\vp_1-(\pi_2\circ\tilde{\Up})^*\vp_2=(id_{M_1})^*\vp_1-\Up^*\vp_2=\vp_1-\Up^*\vp_2.$$
\end{proof}

\section{$G_2$ Vector Fields, Rochesterian $1$-Forms and Rochesterian Vector Fields}
Let $M$ be a $7$-manifold with a $G_2$-structure. Recall that there is an action of the Lie group $G_2$ on the algebra of differential forms on $M$ from which we obtain decompositions of each space of $k$-forms on $M$ into irreducible $G_2$-representations. In particular, we can decompose the space of $2$-forms into the direct sum of a seven-dimensional representation and a fourteen-dimensional representation, denoted from here on by $\Om^2_{7}$ and $\Om^2_{14}$ respectively; it is well known that $\Om^2_{7}=\{X\lrcorner\vp:X\in\X(M)\}$ where $\X(M)$ is the space of vector fields on $M$ (see for example \cite{FeGr, Jo, Kari, Sa}).

\begin{dfn}
Let $(M,\vp)$ be a manifold with closed $G_2$-structure.
\begin{enumerate}
    \item We define a \emph{Rochesterian $1$-form} $\al$ to be any $1$-form on $M$ such that $\d\al\in\Om^2_7(M)$. We denote the set of Rochesterian $1$-forms on $M$ by $\Om^1_{Roc}(M)$.
    \item For a Rochesterian $1$-form $\al$, the vector field $X_{\al}$ satisfying $X_{\al}\lrcorner\vp=\d\al$ will be called a \emph{Rochesterian vector field}, and the set of Rochesterian vector fields on $M$ will be denoted by $\X_{Roc}(M)$.
    \item A vector field $X$ is called a \emph{$G_2$-vector field} if the flow induced by $X$ preserves the $G_2$-structure; equivalently, $X$ is a $G_2$-vector field if $\L_X\vp=0$. $\X_{G_2}(M)$ will denote the set of all $G_2$-vector fields.
\end{enumerate}
\end{dfn}

That $\X_{G_2}$, $\Om^1_{Roc}$ and $\X_{Roc}$ are vector spaces follows immediately from the linearity properties of $\d$ and the interior product. Now, as in the symplectic case, we have the useful fact that $X$ is a $G_2$-vector field if and only if $\d(X\lrcorner\vp)=0$ since $\d\vp=0$ implies that $\L_X\vp=\d(X\lrcorner\vp)+X\lrcorner\d\vp=\d(X\lrcorner\vp)$. In contrast to the symplectic case, the map $\tilde{\vp}:\X(M)\to\Om^2(M)$ given by $\tilde{\vp}(X)=X\lrcorner\vp$ cannot be an isomorphism; however, by the nondegeneracy condition on the $3$-form $\vp$, we do have that $\tilde{\vp}$ is injective, so for a given Rochesterian $1$-form $\al$, the associated Rochesterian vector field $X_{\al}$ is unique. Indeed, the existence of nontrivial Rochesterian vector fields, and hence Rochesterian $1$-forms, is quite a bit more delicate than that of Hamiltonian vector fields as is seen as a consequence of the following theorem (see \cite[Theorem 2.4]{ACS} for the original form of the statement and proof of this theorem):

\begin{thm}
Let $M$ be a closed manifold, and let $\vp$ be a closed $G_2$-structure on $M$. Then $X\lrcorner\vp$ is exact if and only if $X$ is the zero vector field; therefore, there are no nontrivial Rochesterian vector fields on a closed manifold $M$ with closed $G_2$-structure $\vp$.
\end{thm}

\begin{proof}
If $X$ is the zero vector field, then $X\lrcorner\vp=0$ at every point; in this case, we can take $X\lrcorner\vp=\d f$ where $f:M\to \R$ is the constant function $f(p)=0$ for all $p\in M$. Conversely, assume that $X$ is an arbitrary vector field such that $X\lrcorner\vp$ is exact. Then there exists some $1$-form $\al$ such that $X\lrcorner\vp=\d\al$. Using the $G_2$-metric defined by $\vp$, we have that
\begin{equation*}
\begin{split}
6\langle X, X\rangle dvol_M&=(X\lrcorner\vp)\w(X\lrcorner\vp)\w\vp \\
&=\d\al\w\d\al\w\vp=\d(\al\w\d\al\w\vp)\\
\end{split}
\end{equation*}
From here, we find that, since $\pd M=\emptyset$, an application of Stokes' Theorem yields
\begin{equation*}
\begin{split}
0&\leq 6||X||^2_{L^2}vol(M)=\int_M6\langle X, X\rangle dvol_M \\
&=\int_M\d(\al\w\d\al\w\vp)=\int_{\pd M}\al\w\d\al\w\vp=0
\end{split}
\end{equation*}
Since $vol(M)\neq 0$, we must have that $||X||_{L^2}=0$ proving that $X=0$ as desired.
\end{proof}

\begin{rem}
As pointed out by an anonymous referee, there is another instance of nonexistence, this time, for $G_2$ vector fields. Specifically, in the case of a compact torsion-free $G_2$-structure, i. e., a $G_2$-manifold, $G_2$ vector fields, being by definition Killing vector fields, will be parallel since $G_2$-manifolds are Ricci flat. In the case of a nontrivial $G_2$ vector field, there would necessarily be a reduction in the holonomy to a proper subgroup of $G_2$. Thus there are no nontrivial $G_2$ vector fields on a compact irreducible $G_2$-manifold.
\end{rem}

Hence, we assume from now on that $M$ is either noncompact or that $M$ is compact with nonempty boundary; if $M$ happens to be a compact $G_2$-manifold with nonempty boundary, we will further assume that $M$ is reducible. As for specific examples of Rochesterian vector fields, we first have the trivial $G_2$-manifold $(\R^7,\vp_0)$ where simple calculations show that every coordinate vector field is a Rochesterian vector field. For a second, nontrivial example, we first recall from \cite{CST} that if $X$ is a $3$-dimensional manifold, then $(T^*X\t\R,\vp=\Re\Om+\om\w\d t)$ is a $7$-manifold with closed $G_2$-structure where $\Om$ is a certain complex $3$-form and $\om$ is the tautological $2$-form on $T^*X$; then the vector field $\frac{\pd}{\pd t}$ is Rochesterian with an associated Rochesterian $1$-form given by the tautological $1$-form $\al$ on $T^*X$ (see \cite{CST} for more information on this construction).

\begin{thm}
Every Rochesterian vector field on a manifold $M$ with closed $G_2$-structure $\vp$ is a $G_2$-vector field. If every closed form in $\Om^2_7(M)$ is exact, then the spaces $\X_{Roc}(M)$ and $\X_{G_2}(M)$ coincide.
\end{thm}

\begin{proof}
The first statement follows immediately from the definitions. Next, for a $G_2$-vector field $X$, $X\lrcorner\vp\in\Om^2_7(M)$ is closed, so, by assumption, there exists a $1$-form $\al$ with $X\lrcorner\vp=\d\al$.
\end{proof}

\begin{cor}
If $H^2(M)=\{0\}$, then every $G_2$-vector field on a manifold with closed $G_2$-structure is a Rochesterian vector field.
\end{cor}

\begin{prop}
For any $G_2$-vector fields $X_1$, $X_2$, there exists a $1$-form $\al$ such that $[X_1,X_2]\lrcorner\vp=\d\al$.
\end{prop}

\begin{proof}
\begin{equation*}
\begin{split}
[X_1,X_2]\lrcorner\vp&=\L_{X_1}(X_2\lrcorner\vp)-X_2\lrcorner(\underbrace{\L_{X_1}\vp}_{=0}) \\
&=\L_{X_1}(X_2\lrcorner\vp)=\d(X_1\lrcorner X_2\lrcorner\vp)+X_1\lrcorner(\underbrace{\d(X_2\lrcorner\vp)}_{=0})\\
&=\d(\vp(X_2,X_1,\cdot))
\end{split}
\end{equation*}
Thus, $[X_1,X_2]$ is a Rochesterian vector field with an associated $1$-form given by $\vp(X_2,X_1,\cdot)$.
\end{proof}
Thus, we have the following inclusions of \emph{Lie algebras}: $$(\X_{Roc}(M),[\cdot,\cdot])\subseteq(\X_{G_2}(M),[\cdot,\cdot])\subseteq(\X(M),[\cdot,\cdot]).$$

For a Rochesterian $1$-form $\al$, the assignment $\al\mapsto X_{\al}$ where $X_{\al}$ is the unique associated Rochesterian vector field is linear. We now equip $\Om^1_{Roc}(M)$ with a bracket as follows: for $\al,\be\in\Om^1_{Roc}(M)$, define $\{\al,\be\}=\vp(X_{\al},X_{\be},\cdot)$. Then $\{\al,\be\}\in \Om^1_{Roc}(M)$ with Rochesterian vector field given by $[X_{\be},X_{\al}]$ since $$\d(\{\al,\be\})=\d(\vp(X_{\al},X_{\be},\cdot))=[X_{\be},X_{\al}]\lrcorner\vp.$$

\begin{rem}
This bracket is a specific case of the \emph{semibracket} defined in \cite{BHR} for the general multisymplectic setting, and a proof of the following result in this more general setting can be found in \cite[Proposition 3.7]{BHR}.
\end{rem}

\begin{prop}
For any $\al,\be,\ga\in\Om^1_{Roc}(M)$,
$$\{\al,\{\be,\ga\}\}+\{\be,\{\ga,\al\}\}+\{\ga,\{\al,\be\}\}=\d(X_{\al}\lrcorner X_{\be}\lrcorner\d\ga)$$
\end{prop}

\begin{proof}
Let $\al,\be,\ga\in \Om^1_{Roc}(M)$ with associated Rochesterian vector fields $X_{\al}$, $X_{\be}$ and $X_{\ga}$ respectively. Then we have the following:
\begin{equation*}
\begin{split}
\{\al&,\{\be,\ga\}\}+\{\be,\{\ga,\al\}\}+\{\ga,\{\al,\be\}\}=\{\al,\{\be,\ga\}\}-\{\be,\{\al,\ga\}\}-\{\{\al,\be\},\ga\}\\
&=X_{\al}\lrcorner X_{\{\be,\ga\}}\lrcorner\vp-X_{\be}\lrcorner X_{\{\al,\ga\}}\lrcorner\vp-X_{\{\al,\be\}}\lrcorner X_{\ga}\lrcorner\vp\\
&=X_{\al}\lrcorner\d\{\be,\ga\}-X_{\be}\lrcorner\d\{\al,\ga\}+[X_{\al},X_{\be}]\lrcorner\d \ga\\
&=X_{\al}\lrcorner\d(X_{\ga}\lrcorner X_{\be}\lrcorner\vp)-X_{\be}\lrcorner\d(X_{\ga}\lrcorner X_{\al}\lrcorner\vp)+[X_{\al},X_{\be}]\lrcorner\d \ga\\
&=-X_{\al}\lrcorner\d(X_{\be}\lrcorner X_{\ga}\lrcorner\vp)+X_{\be}\lrcorner\d(X_{\al}\lrcorner X_{\ga}\lrcorner\vp)+[X_{\al},X_{\be}]\lrcorner\d \ga\\
&=-X_{\al}\lrcorner\d(X_{\be}\lrcorner\d \ga)+X_{\be}\lrcorner\d(X_{\al}\lrcorner\d \ga)+[X_{\al},X_{\be}]\lrcorner\d \ga\\
&=-X_{\al}\lrcorner\d(X_{\be}\lrcorner\d \ga)+X_{\be}\lrcorner\d(X_{\al}\lrcorner\d \ga)+\L_{X_{\al}}(X_{\be}\lrcorner\d \ga)-X_{\be}\lrcorner(\L_{X_{\al}}\d \ga)\\
&=-X_{\al}\lrcorner\d(X_{\be}\lrcorner\d \ga)+X_{\be}\lrcorner\d(X_{\al}\lrcorner\d \ga)+X_{\al}\lrcorner\d(X_{\be}\lrcorner\d \ga)\\
&+\d(X_{\al}\lrcorner X_{\be}\lrcorner\d \ga)\underbrace{-X_{\be}\lrcorner(X_{\al}\lrcorner\d\d \ga)}_{=0}-X_{\be}\lrcorner\d(X_{\al}\lrcorner\d \ga)\\
&=X_{\al}\lrcorner\d(X_{\be}\lrcorner\d \ga)-X_{\al}\lrcorner\d(X_{\be}\lrcorner\d \ga)+X_{\be}\lrcorner\d(X_{\al}\lrcorner\d \ga)-X_{\be}\lrcorner\d(X_{\al}\lrcorner\d \ga)+\d(X_{\al}\lrcorner X_{\be}\lrcorner\d \ga)\\
&=\d(X_{\al}\lrcorner X_{\be}\lrcorner\d \ga)
\end{split}
\end{equation*}
\end{proof}

While we do not have a Lie algebra structure on $\Om^1_{Roc}(M)$, we do, as noted above, still have a linear transformation $\Phi:\Om^1_{Roc}(M)\to\X_{Roc}(M)$. Assume that $\Phi(\al)=X_{\al}=0$, then $0=X_{\al}\lrcorner\vp=\d\al$ which implies that $\al$ is a closed $1$-form. Hence, Rochesterian vector fields are uniquely defined by their Rochesterian $1$-forms, up to the addition of a closed $1$-form.

\begin{thm}
\begin{enumerate}
    \item Given two Rochesterian $1$-forms $\al_1,\al_2\in\Om^1_{Roc}(M)$, $\{\al_1,\al_2\}\in\ker\Phi$ if and only if $\d\al_1$ is constant along the flow lines of $X_{\al_2}$ if and only if $\d\al_2$ is constant along the flow lines of $X_{\al_1}$.
    \item Let $\psi:(M,\vp)\to(M',\vp')$ be a diffeomorphism. Then $\psi$ is a $G_2$-morphism if and only if $\psi^*(\{\al,\be\})=\{\psi^*\al,\psi^*\be\}$ for all $\al,\be\in\Om^1_{Roc}(M')$.
\end{enumerate}
\end{thm}

\begin{proof}
\begin{enumerate}
    \item Again, we show only the first equivalence since the second equivalence follows similarly. From the definition of the bracket, we have $$\{\al_1,\al_2\}=\vp(X_{\al_1},X_{\al_2},\cdot)=X_{\al_2}\lrcorner(X_{\al_1}\lrcorner\vp)$$ $$=X_{\al_2}\lrcorner\d\al_1=\L_{X_{\al_2}}\al_1-\d(X_{\al_2}\lrcorner\al_1).$$ From this, we see that $\d\{\al_1,\al_2\}=\d\L_{X_{\al_2}}\al_1=\L_{X_{\al_2}}(\d\al_1)$. Then $\{\al_1,\al_2\}\in\ker\Phi$ if and only if $\L_{X_{\al_2}}(\d\al_1)=0$.
    \item First, assume that $\psi$ is a $G_2$-morphism, and note that for $p\in M$, we have the maps
    \begin{equation*}
    \begin{split}
    &\d\psi_p:T_pM\to T_{\psi(p)}M'\\
    &\psi^*_p:T^*_{\psi(p)}M'\to T^*_pM\\
    &\d\psi^{-1}_{\psi(p)}=(\d\psi_p)^{-1}:T_{\psi(p)}M'\to T_pM
    \end{split}
    \end{equation*}
    Since $\psi$ is a $G_2$-morphism, $\psi^*\vp'=\vp$ and $(\psi^{-1})^*\vp=\vp'$, so by definition, for $p\in M$, we then get the following equivalent equations
    \begin{equation*}
    \begin{split}
    \vp_p(\cdot,\cdot,\cdot)&=\psi^*_p(\vp'_{\psi(p)})(\cdot,\cdot,\cdot)=\vp'_{\psi(p)}(\d\psi_p\cdot,\d\psi_p\cdot,\d\psi_p\cdot)\\
    \vp'_{\psi(p)}(\cdot,\cdot,\cdot)&=(\psi^{-1}_{\psi(p)})^*(\vp_p)(\cdot,\cdot,\cdot) =\vp_p(\d\psi^{-1}_{\psi(p)}\cdot,\d\psi^{-1}_{\psi(p)}\cdot,\d\psi^{-1}_{\psi(p)}\cdot)
    \end{split}
    \end{equation*}
    Thus, we calculate for a Rochesterian $1$-form $\al\in\Om^1_{Roc}(M')$ and vector fields $Y,Z$ on $M$,
    \begin{equation*}
    \begin{split}
    (X_{\psi^*\al}\lrcorner\vp)_p(Y_p,Z_p)&=\d(\psi^*\al)_p(Y_p,Z_p)\\
    &=\psi^*_p(\d\al_{\psi(p)})(Y_p,Z_p)\\
    &=\psi^*_p((X_{\al}\lrcorner\vp')_{\psi(p)})(Y_p,Z_p)\\
    &=\psi^*_p(\vp'_{\psi(p)}((X_{\al})_{\psi(p)},\cdot,\cdot))(Y_p,Z_p)\\ 
    &=\vp'_{\psi(p)}((X_{\al})_{\psi(p)},\d\psi_pY_p,\d\psi_pZ_p)\\ 
    &=\vp_p(\d\psi^{-1}_{\psi(p)}(X_{\al})_{\psi(p)},\d\psi^{-1}_{\psi(p)}(\d\psi_pY_p), \d\psi^{-1}_{\psi(p)}(\d\psi_pZ_p))\\ &=\vp_p(\d\psi^{-1}_{\psi(p)}(X_{\al})_{\psi(p)},Y_p,Z_p)
    \end{split}
    \end{equation*}
    that is, $(X_{\psi^*\al})_p=\d\psi^{-1}_{\psi(p)}(X_{\al})_{\psi(p)}$. Hence we find that
    \begin{equation*}
    \begin{split}
    (\psi^*\{\al,\be\})_p(Y_p)&=(\psi^*(\vp'(X_{\al},X_{\be},\cdot)))_p(Y_p)\\
    &=\psi^*_p(\vp'_{\psi(p)}((X_{\al})_{\psi(p)},(X_{\be})_{\psi(p)},\cdot))(Y_p)\\
    &=\vp'_{\psi(p)}((X_{\al})_{\psi(p)},(X_{\be})_{\psi(p)},\d\psi_pY_p)\\
    &=\vp_p(\d\psi^{-1}_{\psi(p)}(X_{\al})_{\psi(p)},\d\psi^{-1}_{\psi(p)}(X_{\be})_{\psi(p)},\d\psi^{-1}_{\psi(p)}(\d\psi_pY_p))\\ &=\vp_p((X_{\psi^*\al})_p,(X_{\psi^*\be})_p,Y_p)\\
    &=\vp_p((X_{\psi^*\al})_p,(X_{\psi^*\be})_p,\cdot)(Y_p)=\{\psi^*\al,\psi^*\be\}_p(Y_p)
    \end{split}
    \end{equation*}
    Conversely, assume that $\psi^*(\{\al,\be\})=\{\psi^*\al,\psi^*\be\}$ for all $\al,\be\in\Om^1_{Roc}(M')$. Then, for any $\al,\be\in\Om^1_{Roc}(M')$ we have
    \begin{equation*}
    \begin{split}
    \psi^*(\{\al,\be\})&=\psi^*(\vp'(X_{\al},X_{\be},\cdot))=\psi^*(X_{\be}\lrcorner X_{\al}\lrcorner\vp')=\psi^*(X_{\be}\lrcorner\d\al)\\
    &=\psi^*(\d\al(X_{\be},\cdot))=\d\al(X_{\be},\d\psi\cdot)=\d\al(\d\psi(\d\psi^{-1}X_{\be}),\d\psi\cdot)\\
    &=(\psi^*\d\al)(\d\psi^{-1}X_{\be},\cdot)=(\d\psi^{-1}X_{\be})\lrcorner(\psi^*\d\al)
    \end{split}
    \end{equation*}
    and
    \begin{equation*}
    \{\psi^*\al,\psi^*\be\}=\vp(X_{\psi^*\al},X_{\psi^*\be},\cdot)=X_{\psi^*\be}\lrcorner\d(\psi^*\al)=X_{\psi^*\be}\lrcorner(\psi^*\d\al)
    \end{equation*}
    which, by our hypothesis, yields that $\d\psi^{-1}X_{\be}=X_{\psi^*\be}$ for any $\be\in\Om^1_{Roc}(M')$. Then for any $\al\in\Om^1_{Roc}(M')$, any vector fields $Y,Z\in\X(M)$ and $p\in M$,
    \begin{equation*}
    \begin{split}
    (X_{\psi^*\al}\lrcorner\vp)_p(Y_p,Z_p)&=\d(\psi^*\al)_p(Y_p,Z_p)=(\psi^*\d\al)_p(Y_p,Z_p)\\
    &=\psi^*_p(X_{\al}\lrcorner\vp')_p(Y_p,Z_p)\\
    &=\vp'_{\psi(p)}((X_{\al})_{\psi(p)},\d\psi_pY_p,\d\psi_pZ_p)\\
    &=\vp'_{\psi(p)}(\d\psi_p(\d\psi^{-1}_{\psi(p)}(X_{\al})_{\psi(p)}),\d\psi_pY_p,\d\psi_pZ_p)\\
    &=(\psi^*\vp')_{\psi(p)}(\d\psi^{-1}_{\psi(p)}(X_{\al})_{\psi(p)},Y_p,Z_p)\\
    &=(\psi^*\vp')_{\psi(p)}((X_{\psi^*\al})_p,Y_p,Z_p)
    \end{split}
    \end{equation*}
    \noindent
    Thus $X_{\psi^*\al}\lrcorner\vp=X_{\psi^*\al}\lrcorner\psi^*\vp'$ which implies that $\vp=\psi^*\vp'$ as desired.
\end{enumerate}
\end{proof}

\end{document}